\documentclass[12pt, a4paper]{amsart}
\usepackage[english]{babel}
\usepackage{amsmath,enumerate, amsfonts, amssymb,amsthm, xypic}
\usepackage[dvips]{graphics} 
\usepackage{color}
\usepackage{graphicx}
\usepackage{hyperref}
\usepackage[hyperpageref]{backref}

\newtheorem{thm}{Theorem}[section]
\newtheorem{lemma}[thm]{Lemma}
\newtheorem{prop}[thm]{Proposition}
\newtheorem{cor}[thm]{Corollary}

\theoremstyle{definition}
\newtheorem{defi}[thm]{Definition}

\newtheorem{ex}[thm]{Example}
\newtheorem{rmk}[thm]{Remark}
\newtheorem{Q}[thm]{Question}

\DeclareMathOperator{\R}{\mathbf R}

\DeclareMathOperator{\N}{\mathbf N}

\DeclareMathOperator{\Ima}{Im}

\DeclareMathOperator{\graph}{graph}

\DeclareMathOperator{\eps}{\epsilon}
\DeclareMathOperator{\Pui}{\mathbb R((t^{1/\infty}))}
\DeclareMathOperator{\wR}{\widetilde \R}

\DeclareMathOperator{\wY}{\widetilde Y}
\DeclareMathOperator{\ord}{ord}

\title{Continuous mappings between spaces of arcs}
\author{Goulwen Fichou and Masahiro Shiota}
\thanks{The first author has been supported by the ANR project ANR-08-JCJC-0118-01.}
\address{IRMAR (UMR 6625), Universit\'e de Rennes 1, Campus de
  Beaulieu, 35042 Rennes Cedex, France}
\address{Graduate School of Mathematics, Nagoya University, Chikusa, Nagoya, 
464-8602, Japan}
\date\today

\begin{document}

\begin{abstract} A blow-analytic homeomorphism is an arc-analytic subanalytic homeomorphism, and therefore it induces a bijective mapping between spaces of analytic arcs. We tackle the question of the continuity of this induced mapping between the spaces of arcs, giving a positive and a negative answer depending of the topology involved. We generalise the result to spaces of definable arcs in the context of o-minimal structures, obtaining notably a uniform continuity property.
\end{abstract}
\maketitle

\section*{Introduction}

The blow-analytic equivalence between real analytic function germs \cite{kuo} is an interesting counterpart in the real setting to the topological equivalence between complex analytic function germs. For $f,g:(\mathbb R^n,0)\to(\mathbb R,0)$ analytic function germs, we say that
$f$ and $g$ are blow-analytically equivalent if there exists a blow-analytic homeomorphism germ $\phi:(\mathbb R^n,0) \to (\mathbb R^n,0)$ such that $f=g\circ\phi$. A homeomorphism $\phi:U\to V$ between open subsets $U$ and $V$ of $\mathbb R^n$ is called a blow-analytic homeomorphism if there exist two finite sequences of blowings-up along smooth analytic centres $\pi:M\to U$ and $\sigma:N\to V$ and an analytic diffeomorphism $\Phi:M\to N$ such that $\phi\circ\pi=\sigma\circ\Phi$.

If the definition of blow-analytic equivalence via sequences of blowings-up makes it difficult to study, it has also very nice properties (cf. \cite{FP} for a survey). In particular a blow-analytic homeomorphism $\phi: U\to V$ is arc-analytic \cite{KK}, namely if $\gamma \in \mathbb R \{t\}^n$ is a n-uplet of convergent power series on a neighbourhood of $0\in \mathbb R$ with $\gamma (0)\in U$, then $\phi \circ \gamma$ is analytic.
In particular a blow-analytic homeomorphism $\phi$ induces a bijective mapping $\phi_*$ between the spaces of analytic arcs at the origin of $\mathbb R^n$. The nice behaviour of the blow-analytic equivalence with respect to arcs and more generally spaces of arcs have already produced very interesting invariants (such as the Fukui invariants \cite{fukui}, zeta functions \cite{KP1,fichou}), a complete classification in dimension two \cite{KP2}, or explained some relations with respect to bi-Lipschitz property \cite{FKP}.

The first question we address in this paper is very natural in this context: given a blow-analytic homeomorphism $\phi$, is the induced mapping $\phi_*$ between the spaces of analytic arcs continuous?
It is natural to hope that such a homeomorphism induces a homeomorphism between arcs, even if the definition of a blow-analytic homeomorphism via sequences of blowings-up makes it difficult to handle directly. We offer in this paper two answers to this question.

A first answer is that the induced mapping is not continuous, even at the level of truncated arcs, when we considered $\mathbb R \{t\}^n$ endowed with the product topology. We show the existence of a counter-example in dimension two in section \ref{c-ex}, where the sequence of blowings-up consists simply of the blowing-up $\pi:M\to \mathbb R^2$ at the origin of $\mathbb R^2$. 

A second answer is that the induced mapping is continuous... if $\mathbb R \{t\}^n$ is endowed with the $t$-adic topology (cf. Theorem \ref{main})! The result is actually a simple consequence of the H\"older property of subanalytic maps. However this question has a natural generalisation in the context of o-minimal structures over a real closed field where such a H\"older property is no longer available. Nevertheless, the tameness of an o-minimal structure should guaranty that the continuity of a mapping over a real closed field continues to hold when we naturally extend the field to another real closed field, and the mapping to a mapping over the extended field. And actually, if the H\"older property suffices to obtain the continuity in the case of subanalytic mappings over real numbers, a generalisation of \L ojasiewicz Inequality to locally closed definable sets (Proposition \ref{L-loc}) enable to control the behaviour of arcs in the o-minimal setting in order to keep the continuity at the level of spaces of definable arcs.

We propose moreover another approach, more natural in the non necessary locally closed case (e.g. a bijection coming from a resolution of the singularities as in Example \ref{exam}), and following a geometric approach parallel to the model-theoretic point of view developed in \cite{Coste}. In particular, we study more in details in section \ref{trans} the transport of properties between the initial o-minimal structure other a given real closed field and the new o-minimal structure on the real closed field of germs of definable arcs at the origin. We obtain moreover in a very simple way a uniform continuity property in Proposition \ref{prop-fin}.

\vskip 5mm

{\bf Acknowledgements.} The first author wish to thank K. Kurdyka, O. Le Gal, M. Raibaut and S. Randriambololona for valuable remarks.
\section{Blow-analytic homeomorphisms and continuity}\label{sect-bah}

If a blow-analytic homeomorphism induces a continuous mapping at the level of spaces of arcs considered with the $t$-adic topology (cf. Theorem \ref{main}), we prove the existence of a blow-analytic homeomorphism which does not induce a continuous mapping at the level of spaces of arcs when we considered it with the product topology. The counter-example is produced in section \ref{c-ex}.

\subsection{Blow-analytic homeomorphisms}

\begin{defi} Let $U$ and $V$ be open subsets of $\mathbb R^n$. 
A homeomorphism $\phi:U\to V$ is called a {\it blow-analytic homeomorphism} if there exist two finite sequences of blowings-up along smooth analytic centres $\pi:M\to U$ and $\sigma:N\to V$ and an analytic diffeomorphism $\Phi:M\to N$ such that $\phi\circ\pi=\sigma\circ\Phi$.
\end{defi}

Denote by $\mathbb R \{t\}$ the one-variable convergent power series ring and $\mathfrak m$ its maximal ideal. We consider in this section $\mathbb R\{t \}$ equipped either with the $t$-adic topology or with the product topology.
We regard $\mathbb R \{ t\}^n$ as the family of analytic curve germs $c:[0;\,\epsilon)\to\mathbb R^n$, with $\epsilon>0\in\mathbb R$, at $0$ and let $\mathcal A_{0}(\mathbb R^n)$ denote those curve germs $c$ with $c(0)=0$.
Set 
$$\mathcal A_{U}(\mathbb R^n)=\{c\in\mathbb R \{t\}^n:c(0)\in U\}$$
for an open subset $U$ of $\mathbb R^n$. 
We identify $\mathcal A_{U}(\mathbb R^n)$ with $U\times\mathcal A_{0}(\mathbb R^n)$ by the correspondence $$\mathcal A_{U}(\mathbb R^n)\ni c\to(c(0),c-c(0))\in U\times\mathcal A_{0}(\mathbb R^n).$$

\subsection{$t$-adic topology}\label{res-t}

Let $h: (\mathbb R^n,0) \to (\mathbb R^n,0)$ be an arc-analytic mapping germ \cite{KK}, namely any analytic arc $\gamma: (-\epsilon;\epsilon) \to \mathbb R^n$, where $\epsilon>0$, with $\gamma (0)=0$ is sent by $h$ into an analytic arc $h\circ \gamma :(-\epsilon';\epsilon') \to \mathbb R^n$ for some $\epsilon'>0$. Then $h$ defines a mapping from $\mathcal A_{0}(\mathbb R^n)$ to $\mathcal A_{0}(\mathbb R^n)$, denoted by $h_*$ in the sequel.
We considered $\mathcal A_{0}(\mathbb R^n)$ with its $t$-adic topology.

\begin{rmk}\label{rmk-bah} A blow-analytic homeomorphism is arc-analytic since any analytic arc $c:[0;\,\epsilon)\to U$ may be lifted via a sequence of blowings-up along smooth analytic centres $\pi:M\to U$ to an analytic arc $d:[0;\,\epsilon)\to M$ such that $c=\pi \circ d$ (cf. \cite{FP} section 5 for example). In particular a blow-analytic homeomorphism $\phi:U\to V$ induces a mapping 
$$\phi_*:\mathcal A_{U}(\mathbb R^n)\ni c\to\phi\circ c\in \mathcal A_{V}(\mathbb R^n)$$
Note that this mapping is moreover bijective.
\end{rmk}

\begin{thm}\label{main} Let $h: (\mathbb R^n,0) \to (\mathbb R^n,0)$ be a
  subanalytic homeomorphism. Assume $h$ and $h^{-1}$ are arc analytic. The induced mapping $h_*:\mathcal A_{0}(\mathbb R^n) \to  \mathcal A_{0}(\mathbb R^n)$ is a uniformly continuous homeomorphism with respect to the $t$-adic topology.
\end{thm}

\begin{proof} As a consequence of \L ojasiewicz inequality (\cite{BM}, Theorem 6.4), a subanalytic homeomorphism $h$ is H\"older, and so there exist $\alpha \in \mathbb Q$, with $\alpha \in ]0,1]$, such that
$$|h(x)-h(y)| \leq c|x-y|^{\alpha}$$
for $x,y$ close to 0, where $c>0$.
In particular, if $\gamma, \delta: [0,\epsilon) \to (\R^n,0)$ are analytic arcs, then
$$\ord (h\circ \gamma (t)- h\circ \delta (t)) \geq \alpha \ord (\gamma (t)- \delta (t)),$$
and this gives the continuity with respect to the t-adic topology.
\end{proof}

Theorem \ref{main} extends to a global version on compact analytic manifolds. Let $M\subset \mathbb R^n$ be a compact analytic manifold, and denote by $\mathcal A(M)$ those arcs in $\mathbb R\{t\}^n$ with origin in $M$. We consider $\mathcal A(M)$ endowed with the topology induced by that of $\mathbb R\{t\}^n$.

\begin{thm}\label{thm-comp} Let $M\subset \mathbb R^m$ and $N\subset \mathbb R^n$ be compact analytic manifolds and $h:M \to N$ be a subanalytic homeomorphism. Assume that $h$ and $h^{-1}$ are arc analytic. The induced mapping $h_*: \mathcal A(M) \to \mathcal A(N)$ is a uniformly continuous homeomorphism.
\end{thm}

If the ambient spaces are no longer compact, we keep the continuity of the induced mapping between analytic arcs since continuity is a local property:

\begin{thm}\label{thm-R} Let $h: \mathbb R^n \to \mathbb R^n$ be a
  subanalytic homeomorphism. Assume $h$ and $h^{-1}$ are arc analytic. Then $h_*:\mathbb R\{t\}^n \to \mathbb R\{t\}^n$ is a homeomorphism.
\end{thm}


\subsection{Product topology}\label{c-ex}
For a blow-analytic homeomorphism $\phi$ as in Remark \ref{rmk-bah}, it is natural to hope that the induced mapping $\phi_*$ is a homeomorphism when we considered the product topology. However, it is not difficult to find a counter-example!

In the sequel, we focus therefore on some particular blow-analytic homeomorphisms, namely those blow-analytic homeomorphisms which induce a blow-analytic equivalence between some analytic functions (we refer to \cite{FP} for similar distinctions between different kind of blow-analytic homeomorphisms).

\begin{defi} Let $f,g:(\mathbb R^n,0)\to(\mathbb R,0)$ be analytic function germs. 
We call $f$ and $g$ {\it blow-analytically equivalent} if there exist open subsets $U$ and $V$ of $\mathbb R^n$ containing 0 and a blow-analytic homeomorphism $\phi:U\to V$ such that $\phi(0)=0$ and $f=g\circ\phi_0$, where $\phi_0$ denotes the germ of $\phi$ at 0. 
We call $\phi_0:(\mathbb R^n,0)\to(\mathbb R^n,0)$ a {\it blow-analytic homeomorphism germ}. 
\end{defi}

The problem we address is then as follows.

\begin{Q}\label{Ques} Let $\phi$ be a blow-analytic homeomorphism germ which realises the blow-analytic equivalence of two non-zero analytic function germs. Is the induced mapping $\phi_*:\mathcal A_{0}(\mathbb R^n)\ni c\to\phi\circ c\in\mathcal A_{0}(\mathbb R^n)$ continuous in the product topology?
\end{Q}
  
In the following we give a negative answer to Question \ref{Ques}.
We prove the existence of a counter-example by an explicit construction. For the counter-example, we fix $n=2$ and define $f,g:(\mathbb R^2,0)\to(\mathbb R,0)$ by $f(x,y)=g(x,y)=y$. We define a family of curves $c_{\epsilon}\in\mathcal A_{0}(\mathbb R^2)$ by 
$$c_\epsilon(t)=(\epsilon t,t^2)$$ 
for $\epsilon\in\mathbb R$. 
Then $c_\epsilon$ converges to $ c_0$ in $\mathcal A_{0}(\mathbb R^2)$ for the product topology as $\epsilon$ goes to $0$.

\begin{figure}
\centering
\includegraphics[height =4cm]{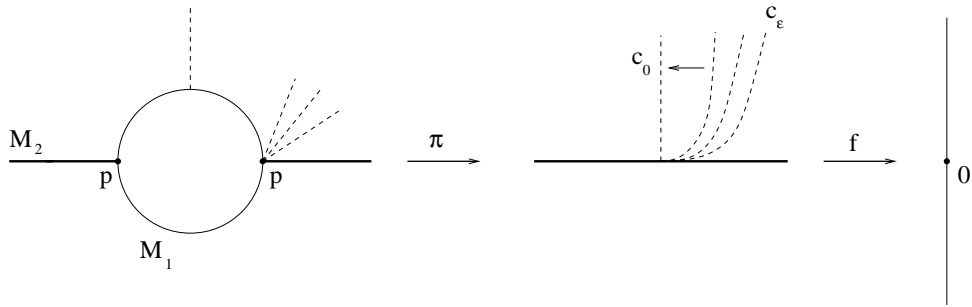}
\caption{Blowing-up and curves.}\label{fig}
\end{figure}

Let $\pi:M\to\mathbb R^2$ denote the blowing-up of $\mathbb R^2$ along center $\{0\}$. We will prove the existence of an analytic diffeomorphism $\Phi$ of $M$ such that:
\begin{enumerate}
\item[(i)] $\Phi$ induces a blow-analytic homeomorphism $\phi$ of $\mathbb R^2$ with $\phi(0)=0$, 
\item[(ii)] $f\circ\phi=f$ and 
\item[(iii)] $(\phi_0)_*(c_\epsilon)\not\to(\phi_0)_*(c_0)$ as $\epsilon\to0$.
\end{enumerate}

Let $M_1=\pi^{-1}(0)$ denote the exceptional divisor of $\pi$ and $M_2$ denote the closure of $\pi^{-1}(\mathbb R\times\{0\}-\{0\})$ in $M$.

\begin{lemma}\label{lemtau} Let $\Phi$ be an analytic diffeomorphism
  of $M$ such that (i) and (ii) hold. Then $\Phi(M_1)=M_1$ and $\Phi(M_2)=M_2$.
\end{lemma}

\begin{proof} Condition (i) is equivalent to the condition $\Phi(M_1)=M_1$ while condition (ii) is equivalent to the condition 
\begin{equation}\tag{ii'}
f\circ\pi\circ\Phi=f\circ\pi
\end{equation}
Then$(f\circ\pi)^{-1}(0)=M_1\cup M_2$ and hence condition (ii') implies $\Phi(M_2)=M_2$.
\end{proof}

Assume that there exists an analytic diffeomorphism $\Phi$ with the conditions (i), (ii) and (iii) satisfied. 
We describe when (iii) holds by a coordinate system. 
Denote by $p$ the intersection point $M_1\cap M_2=\{p\}$ and consider $\pi$ around $p$. Choose a chart $U$ in $M$ such that in the chart $(M,M_1,M_2,p)$ is equal to $(\mathbb R^2,\{0\}\times\mathbb R,\mathbb R\times\{0\},(0,0))$ and the restriction of $\pi$ to $U$ coincides with the mapping 
$$\mathbb R^2\ni(u,v)\to(u,u v)=(x,y)\in\mathbb R^2$$
that we still denote by $\pi$. Note that in this chart $U$, the set $\Ima c_\epsilon$, for $\epsilon\not=0$, is included in $\Ima\pi$ whereas $\Ima c_0$ is not.

The restriction of $\Phi$ to $U$ is of the form $(u\Phi_1(u,v),v\Phi_2(u,v))$ for some analytic functions $\Phi_1$ and $\Phi_2$ on $\mathbb R^2$ which vanish nowhere by Lemma \ref{lemtau}. 

\begin{lemma}\label{chart} The image by $\phi$ of the family of curves $c_{\epsilon
}$, for $\epsilon>0$, is given by
$$\phi_*(c_\epsilon)(t)=(\epsilon t\Phi_1(\epsilon t,\frac{t}{\epsilon}),t^2).$$
\end{lemma}
\begin{proof}
In the chart $U$ considered, we obtain for $\epsilon>0$: 
$$
\aligned
\pi^*(c_\epsilon)(t)=(\epsilon t,\frac{t}{\epsilon}),\ \ \Phi_*\pi^*(c_\epsilon)(t)=(\epsilon t\Phi_1(\epsilon t,\frac{t}{\epsilon}),\frac{t}{\epsilon}\Phi_2(\epsilon t,\frac{t}{\epsilon})),\qquad\ \ \\
\phi_*(c_\epsilon)(t)=\pi_*\Phi_*\pi^*(c_\epsilon)(t)=(\epsilon t\Phi_1(\epsilon t,\frac{t}{\epsilon}),t^2\Phi_1(\epsilon t,\frac{t}{\epsilon})\Phi_2(\epsilon t,\frac{t}{\epsilon})),\qquad\\
f\circ\pi(u,v)=u v. \qquad\qquad\qquad\qquad\qquad\qquad
\endaligned
$$
The last equality together with condition (ii') implies
$$u v\Phi_1(u,v)\Phi_2(u,v)=f\circ\pi\circ\Phi(u,v)=f\circ\pi(u,v)=u v$$
hence  $\Phi_1\Phi_2\equiv1$. In particular $\phi_*(c_\epsilon)(t)=(\epsilon t\Phi_1(\epsilon t,\frac{t}{\epsilon}),t^2)$. 
\end{proof}

The image $\phi_*(c_0)$ of $c_0$ belongs to $\mathfrak m$ because $\phi_*(c_0)(0)=\phi(0)=0$. 
As a consequence, to obtain a negative answer to Question \ref{Ques}, it suffices to find a global analytic diffeomorphism $\Phi$ such that $\epsilon t\Phi_1(\epsilon t,\frac{t}{\epsilon})$ does not converge to any element of $\mathfrak m$ as $\epsilon\to 0$ in $\mathbb R$. 

\begin{rmk}\begin{flushleft}\end{flushleft} 
\begin{enumerate}
\item Note that $\phi_*(c_0)$ does exist, even if we do not know how to describe it.
\item Note that the non-convergence of $\pi^*(c_\epsilon)$ does not imply necessarily the non-convergence of $\phi_*(c_\epsilon)$ whereas the non-convergence of $\epsilon t\Phi_1(\epsilon t,\frac{t}{\epsilon})$ does.
\end{enumerate}
\end{rmk}

Describe $\Phi_1$ around 0 as a convergent power series $\sum_{m,n\in\N}a_{m,n}u^m v^n$ with $a_{m,n} \in \mathbb R$. Then the first component of $\phi_*(c_\epsilon)(t)$ for $\epsilon>0$ is given by
$$\epsilon t\Phi_1(\epsilon t,\frac{t}{\epsilon})\newline=\sum_{m,n\in\N}a_{m,n}t^{m+n+1}\epsilon^{m-n+1}.$$

\begin{lemma} 
If there exist $m$ and $n$ in $\N$ such that $n>m+1$ and $a_{m,n}\not=0$, then $\epsilon t\Phi_1(\epsilon t,\frac{t}{\epsilon})$ does not converge to any element of $\mathfrak m$ as $\epsilon\to 0$. 
\end{lemma}

\begin{proof}
Assume that there exist such pairs $(m,n)$ with $n>m+1$ and $a_{m,n}\not=0$. 
Let $L$ denote the family of all such pairs. Denote by $L_1$ the subfamily of $L$ consisting of $(m,n)$ such that the sum $m+n$ is the smallest possible in $L$ and by $L_2$ the subfamily of $L_1$ of $(m,n)$ such that $n-m$ is the largest possible in $L_1$. 
Then $L_1$ is non-empty and finite, and $L_2$ is non-empty and consists of a unique element, say $(m_0,n_0)$. Then $\epsilon t\Phi_1(\epsilon t,\frac{t}{\epsilon})$ satisfies 
$$\epsilon t\Phi_1(\epsilon t,\frac{t}{\epsilon})-\sum_{n\leq m+1}a_{m,n}t^{m+n+1}\epsilon^{m-n+1}$$
$$=a_{m_0,n_0}t^{m_0+n_0+1}(\epsilon^{m_0-n_0+1}+\delta_{m_0+n_0+1}(\epsilon))+\sum_{k=m_0+n_0+2}^\infty\delta_k(\epsilon)t^k$$ 
where the functions $\delta_k$ are functions in the variable $\epsilon$, and $\delta_{m_0+n_0+1}$ satisfies 
$$|\delta_{m_0+n_0+1}(\epsilon)|\le c|\epsilon||\epsilon^{m_0-n_0+1}|$$ 
for some $c>0\in\mathbb R$ and $\epsilon\not=0$ near 0. 
Therefore the coefficient of $t^{m_0+n_0+1}$ does not converge as $\epsilon\to 0$, so that $\epsilon t\Phi_1(\epsilon t,\frac{t}{\epsilon})$ does not converge to any element of $\mathfrak m$ as $\epsilon\to 0$.
\end{proof}

For simplicity, we strengthen the condition 
$$a_{m,n}\not=0 \textrm{~~for some~~} n>m+1$$ 
to the condition
$$a_{0,n}\not=0 \textrm{~~ for some~~} n>1$$ 
In particular in that case 
$$\Phi_1(0,v)=a_{0,0}+a_{0,1}v+\sum_{n\in\N}a_{0,n}v^n.$$ 
To finish the proof of the negative answer to Question \ref{Ques}, we are going to give an explicit example with $\Phi_1$ satisfying $a_{0,n}\not=0$ for some $n>1$.
Let $\phi: \R^2 \to \R^2$ be defined by
$$\phi(x,y)=(xP(x,y)^{1/4},y)$$
if $(x,y)\neq (0,0)$ and $\phi(0,0)=(0,0)$, where
$$P(x,y)=1+\frac{y^2}{x^2+y^2}.$$ 
Remark that $\phi$ is continuous on $\R^2$ since $0\leq P\leq 2$, and analytic in restriction to $\R^2\setminus\{(0,0)\}$. Note moreover that $\phi$ fix the $x$-axis together with the $y$-axis.

Let us prove first that $\phi$ is a homeomorphism. It suffices to prove that for $y\neq 0$, the one variable function $h(x)=xP(x,y)^{1/4}$ is strictly increasing. Its derivative is given by 
$$h'(x)=P(x,y)^{1/4}(1-\frac{x^2}{2(x^2+y^2)(x^2+2y^2)})$$
and $$\frac{x^2}{2(x^2+y^2)(x^2+2y^2)}\leq \frac{(x/y)^2}{2((x/y)^2+1)((x/y)^2+2)}\leq \frac{1}{4}$$
so that $h'>0$.

We are going to prove that $\phi$ lift to an analytic diffeomorphism $\Phi:M\to M$. Recall that $U$ denotes the chart on $M$ such that $\pi_{|U}$ is given by $\pi_{|U}(u,v)=(u,uv)$ and denote by $V$ the chart such that $\pi_{|V}(u,v)=(uv,v)$. Then $\phi$ lifts to a map $\Phi:M \to M$ such that $\Phi(U)\subset U$ and $\Phi(V)\subset V$, and more precisely the restriction $\Phi_U: U\to U$ of $\Phi$ to $U$ is given by
$$\Phi_U(u,v)=(uP(1,v)^{1/4}, vP(1,v)^{-1/4})$$
whereas $\Phi_V: V\to V$ is equal to
 $$\Phi_V(u,v)=(uP(u,1)^{1/4},v).$$
In particular $\Phi$ is analytic. Moreover $\Phi_{|M_1}$ is a bijection onto $M_1$ since in the chart $U$ for example, where $M_1$ is defined by $u=0$, we have
$$\Phi_U(0,v)=(0,vP(1,v)^{-1/4}).$$
Finally the inverse of $\Phi$ is also analytic by the Jacobian criterion.

As a consequence $\phi$ is a blow-analytic homeomorphism of $\R^2$ (cf. \cite{FKuoP} for close examples), the map $\Phi$ is an analytic diffeomorphism of $M$ satisfying conditions $(i)$ and $(ii)$, and moreover 
$$\Phi_1(0,v)=P(1,v)^{1/4}$$ 
is of the required form so that $(iii)$ is also satisfied.


\section{Set of germs of definable functions}\label{top}
In order to deal with a generalisation of Theorem \ref{main} in the context of o-minimal structures in section \ref{co}, we introduce the real closed field of definable germs of arcs over a given real closed field in section \ref{extension}. We begin with some basic facts about real closed fields and recall
the unique Euclidean topology on a real closed field, following \cite{BCR}.

\subsection{Real closed fields}
An ordered field is a field $F$ equipped with an ordering, namely
there exists a total order relation $\leq$ on $F$ satisfying $x+z\leq
y+z$ if $x\leq y$ and the product of positive elements is
positive. Fields such as $\mathbb Q$ and $\mathbb R$ are ordered fields. 

\begin{ex} There exist several orderings on the field $\mathbb
R(t)$ of real rational fractions. For example, we may choose that $a_nt^n+a_{n+1}t^{n+1}+\cdots+a_mt^m$, with $a_n\neq 0$ is positive if and only if $a_n>0$. In particular $t$ is positive and smaller than any positive real number. The situation is similar with the field of formal power series $\mathbb R((t))$.
\end{ex}

A real field is a field that can be ordered. A real closed field is a
real field that 
does not admit any non trivial real extension. A real
closed field $\R$ is characterised by the fact that $\R$ admits a unique ordering such that the positive elements
coincide with the squares and every polynomial in $\R[X]$ of
odd degree has a root in $\R$. Every ordered field admits a real
closure, unique up to a unique isomorphism. Of course $\mathbb R$ is
real closed. The real closure of $\mathbb Q$ consists of the field of
real algebraic numbers $\mathbb R_{alg}$.

\begin{ex} The field of (formal) real Puiseux series $\Pui=\cup_{q\in \mathbb N^*}\mathbb R((t^{1/q}))$, namely the set of formal expressions
$$\sum_{p=n}^{+\infty} a_pt^{p/q} \textrm{ with } n\in \mathbb Z, ~q\in \mathbb N^*, ~ a_p \in \mathbb R,$$
is real closed. A positive Puiseux series is a series of the form $\sum_{p=n}^{+\infty} a_pt^{p/q}$ with $a_n>0$.  
  
The real closure of the field of rational functions $\mathbb
R(t)$ is the field of algebraic Puiseux series $\Pui_{alg}$, i.e. Puiseux series
algebraic over $\mathbb R(t)$. The field of convergent Puiseux series $\mathbb R \{t^{1/\infty}\}$ is also a real closed field. The following inclusions hold:
$$\mathbb R_{alg}\subset \mathbb R \subset \Pui_{alg}  \subset
\mathbb R \{t^{1/\infty}\} \subset
\Pui $$
\end{ex}

\subsection{Topology}
A real closed field $\R$ induces an Euclidean topology on $\R^n$ as
follows: for $x\in \R^n$, defined
$|x|=\sqrt{x_1^2+\cdots+x_n^2}$. Then the open balls
$$B(x,r)=\{y\in \R^n: |y-x|<r \}$$
for $x\in \R^n$ and $r\in \R$ with $r>0$, form a basis of open
subsets. The Euclidean topology is the unique topology compatible with the real
structure, in the sense that a real closed field admits a unique
ordering. It is called the Euclidean topology of the real closed field.

\begin{ex}\label{ex}\begin{flushleft}\end{flushleft}
\begin{enumerate}
\item The subset of $\Pui_{alg}$
  defined by $$\{\sum_{k\in \mathbb
    N} a_kt^{k/p}:~~ a_k\in \mathbb R,~~p\in \mathbb N^*\}$$ is open. Actually it is an infinite union of open intervals:
$$\cup_{a\in (0;+\infty)\cap \mathbb R} (-a;a).$$
\item As a real closed field, $\Pui_{alg}$ inherits to an Euclidean topology. An element $\gamma$ in $\Pui_{alg}$ of the form $\gamma (t)=\sum_{p\geq m}a_p t^{p/q}$, with $a_p \in \mathbb R$, $q\in \mathbb N^*$, and $a_m \neq 0$ is small in this topology if and only if $m$ is large.
\end{enumerate}
\end{ex}

Note that the Euclidean topology of a real closed field may behave
very differently from that of real numbers: a real closed field is not necessarily connected, and
$[0,1]$ is not compact in $\mathbb R_{alg}$ nor in $\Pui$.

\begin{rmk} On the field of (convergent) power series $\mathbb R\{t\}$ one
  can consider several topologies such as the product topology or
  $t$-adic topology studied in section \ref{sect-bah}. Note that the $t$-adic topology on $\mathbb R\{t\}$ coincides with the restriction to $\mathbb R\{t\}$ of the Euclidean
  topology of the field of convergent Puiseux series $\mathbb R
  \{t^{1/\infty}\}$.
\end{rmk}
\subsection{o-minimal structures expanding $\R$}\label{extension}
We introduce an o-minimal structure expanding a real closed field following \cite{Coste}.

Let $\R$ be a real closed field. 
An o-minimal structure expanding $\R$ is a collection $\mathcal S=(\mathcal S^n)_{n \in \mathbb N}$, where each $\mathcal S^n$ is a set of subsets of the affine space $\R^n$, satisfying the following axioms:
\begin{enumerate}
\item All algebraic subsets of $\R^n$ are in $\mathcal S^n$.
\item For every $n$, $\mathcal S^n$ is a Boolean subalgebra of subsets of $\R^n$.
\item If $A\in \mathcal S^m$ and $B \in \mathcal S^n$, then $A\times B \in \mathcal S^{n+m}$.
\item If $p:\R^{n+1} \to \R^n$ is the projection on the first $n$ coordinates and $A\in \mathcal S^{n+1}$, then $p(A) \in \mathcal S^n$.
\item An element of $\mathcal S^1$ is a finite union of points and open intervals $(a;b)=\{x\in \R : a<x<b \}$, with $a,b \in \R \cup \{\pm \infty \}$.
\end{enumerate}

A set belonging to the collection $\mathcal S$ is called a definable set.

\begin{ex}\begin{flushleft}\end{flushleft}
\begin{enumerate}
\item Let $\R$ be a real closed field. The most simplest o-minimal structure expanding $\R$ is the structure whose definable sets are the semialgebraic sets.
\item A globally subanalytic subset $A$ of $\mathbb R^n$ is a subanalytic subset $A$ of $\mathbb R^n$ which is subanalytic at infinity. Namely, if we embed $\mathbb R^n$ in $\mathbb S^n$ (via some rational regular embedding) and consider $\mathbb S^n$ in $\mathbb R^{n+1}$, then we ask that $A\subset \mathbb R^n \subset \mathbb S^n \subset \mathbb R^{n+1}$ is a subanalytic subset of $\mathbb R^{n+1}$.
The collection of globally subanalytic sets forms an o-minimal structure expanding $\mathbb R$.
\end{enumerate}
\end{ex}

A definable function is a function defined on a definable set whose graph is definable. We will use several times the curve selection lemma in o-minimal structures in the sequel.

\begin{thm}[Curve selection lemma]\label{csl} Let $\R$ be a real
  closed field and let be given an o-minimal structure expanding $\R$. Let $A$ be a definable subset of $\R^n$ and $x\in
  \overline A$. There exists a continuous definable mapping
  $\gamma:[0;1)\to \R^n$ such that $\gamma(0)=x$ and $\gamma
  ((0;1))\subset A$.
\end{thm}

We will also need to reparametrize definable arcs (\cite{vdd}, Exercise (1.9) p 49).

\begin{lemma}\label{repara} Let $\gamma : (0;\epsilon) \to \R$ be a non-constant 
continuous definable function. There 
exist numbers $\epsilon_1$ and $\epsilon_2$ in $\R$ with $\epsilon_1<\epsilon_2$
and a continuous definable bijection $\delta : 
(\epsilon_1;\epsilon_2) \to (0;\epsilon_3)$, with $0<\epsilon_3 < \epsilon$ such that $\gamma \circ \delta (t)=t$ for any $t\in (\epsilon_1;\epsilon_2)$.
\end{lemma}


\subsection{Germs of definable functions}\label{defin}
Given a real closed field and an o-minimal structure expanding it, we introduce its field of germs of continuous definable curves at the origin in section \ref{defin}. We derive an o-minimal structure expanding it and study in section \ref{trans} the transport of some properties from the initial o-minimal structure to the new one.

Let $\R$ be a real closed field. We fix an o-minimal structure 
over $\R$ expanding $\R$. 
Let $\wR$ be the set of germs at
$0\in \R$ of continuous definable functions from $(0;\infty)$ to
$\R$. 

\begin{lemma} 
The set $\wR$ is a real
  closed field.
\end{lemma}

\begin{proof} 
The set $\wR$ is a field. Actually a non zero definable function on a
neighbourhood of $0$ in $(0;\infty)$ nowhere
vanishes on a sufficiently small neighbourhood of 0, so its inverse is well defined and definable.

To prove that $\wR$ is a real closed field, it is sufficient to prove that $\wR[i]=\frac{\wR[Y]}{(Y^2+1)}$ is an algebraically closed field. 
First we prove that the ring $\frac{\wR[Y]}{(Y^2+1)}$ is a field. 
For that it suffices to show that $(Y^2+1)$ is a prime ideal. 
Otherwise, $Y^2+1$ is the product of two polynomial functions because $\wR[Y]$ is a unique factorization ring and hence $Y^2+1$ is of the form $(Y-f_1)\times(Y-f_2)$ and $Y^2-f_1^2$ for some $f_1$ and $f_2$ in $\wR$, which is impossible. 
For $P(t,X)\in \wR[i][X]$, write
$$P(t,X)=\sum_{j=0}^df_j(t)X^j$$
with $f_j\in \wR[i]$. Choose $\epsilon >0$ such that every $f_j$, for $j=0,\ldots,d$, is defined on $(0;\epsilon)$. For $t\in (0;\epsilon)$ fixed, the polynomial $P(t,X)\in \R[i][X]$ admits roots in $\R[i]$ and the map
$$\{(t,x_1,\ldots,x_d)\in (0;\epsilon)\times \R[i]^d : P(t,x)=0\} \to  (0;\epsilon)$$
is finite-to-one. By the o-minimal Hardt triviality Theorem (which is proved in the same way as the semialgebraic case, cf. \cite{BCR}, Theorem 9.3.2), this mapping admits a definable continuous section on a neighbourhood of $0$ in $(0;\epsilon)$. It furnishes a root of $P$ in $\wR[i]$.
\end{proof}

\begin{rmk}\begin{flushleft}\end{flushleft}
\begin{enumerate}
\item Note that $\wR$ is even a field extension of $\R$ by assigning to a number
$x\in \R$ the constant function $\widetilde x$ with image $x$.
\item  A positive element in $\wR$ is a square; in particular a definable continuous function germ $\gamma : (0;\infty) \to \R$ at $0$ is positive in $\wR$ if and only if there exists $\epsilon >0$ in $\R$ such that $\gamma (t) \geq 0$ for $t \in (0;\epsilon)$.
\end{enumerate}
\end{rmk}

\begin{ex}\begin{flushleft}\end{flushleft}
\begin{enumerate}
\item Note that the field of germs of continuous semialgebraic functions defined on intervals of the form $(0;\epsilon)\subset \mathbb R$ is isomorphic (cf. \cite{BCR}) to the field of algebraic Puiseux series $\Pui_{alg}$. The subring of algebraic formal power series $\mathbb R[[t]]_{alg}$ corresponds to the germs at the origin of analytic semialgebraic functions defined on intervals $[0;\epsilon)$.

In particular, if $\R=\mathbb R$ and the o-minimal structure is defined by the semialgebraic sets, then $\wR$ is the field of algebraic
  Puiseux series $\Pui_{alg}$.

Note that the same holds true for any real closed field $\R$ in place of $\mathbb R$.
\item If $\R=\mathbb R$ and the o-minimal structure is given by the globally subanalytic sets, then $\wR$ is the field of convergent Puiseux series $\mathbb R \{t^{1/\infty}\}$.
\end{enumerate}
\end{ex}

For a definable set $X\subset \R^n$, for $n \in \mathbb N$, denote by $\widetilde {X}$ the set of germs at
$0\in \R$ of continuous definable functions from $(0;\infty)$ to
$X$. For $x \in X$, we denote again $\widetilde x\in \widetilde X$ the germ of the constant
function equal to $x$.

\begin{lemma} Let $f:\R \to \R$ be definable. Let $\widetilde f:\wR
  \to \wR$ be defined by $\widetilde f (\gamma (t))=f\circ \gamma
  (t)$ for $\gamma \in \wR$. Then 
$$\graph \widetilde f=\widetilde
  {\graph f}.$$
\end{lemma}

\begin{proof} It suffices to remark that ${\widetilde R}^2=\widetilde {R^2}$ since 
$$\widetilde {\graph f}=\{\gamma: (0;\epsilon) \to R^2 : \gamma \textrm{ continuous definable and }
\Ima 
\gamma \subset \graph f \}\subset \widetilde {R^2}$$
$$\graph \widetilde f=\{(\gamma_1, \gamma_2)\in  {\widetilde R}^2: \gamma_2=f \circ \gamma_1   \} \subset  {\widetilde R}^2$$
\end{proof}

In \cite{Coste}, M. Coste prove as Theorem 5.8 that the collection $\mathcal S$ of subsets of $\widetilde {\mathbf R^n}$ given by $\widetilde X$, for $X\subset \R^n$ definable and $n \in \mathbb N$, together with the fibres of definable families (cf \cite{Coste} Definition 5.7), defines an o-minimal structure expanding $\wR$. These fibres enable typically to add the singleton $\{t\}$ as a definable subset of $\widetilde R$. We give below a geometric proof of a particular case of interest for us,  namely that the collection of subsets of $\widetilde {\mathbf R^n}$ given by $\widetilde X$, for $X\subset \R^n$ definable and $n \in \mathbb N$, are stable under the usual operation in o-minimal structures. 

\begin{prop} The collection of subsets of $\widetilde {\mathbf R^n}$ given by
  $\widetilde X$, for $X\subset \R^n$ definable and $n \in \mathbb N$,
  defines the $0$-definable sets of the o-minimal structure $\mathcal S$.
\end{prop}

\begin{proof} Let $A$ and $B$ are definable sets in $\R^n$. 
What we need to prove is the following:
\begin{enumerate}
\item[(i)] $\widetilde {A \cup B}=\widetilde A \cup \widetilde B$, $\widetilde {A \cap B}=\widetilde A \cap \widetilde B$, $\widetilde {A \setminus B}=\widetilde A \setminus \widetilde B$, $\widetilde {A\times B}=\widetilde A \times \widetilde B$,
\item[(ii)] if $p$ denotes the projection $p:\R^{n+1} \to \R^n$ onto the first $n$ coordinates, and $C\subset \R^{n+1}$ is definable, then $\widetilde {p(C)}=\widetilde p ( \widetilde C)$,
\item[(iii)] $$\widetilde {\{ (x,y,z)\in \R^3 : x=yz\}}=\{ (\alpha,\beta,\gamma)\in \widetilde \R^3 : \alpha=\beta \gamma\},$$
\item[(iv)] $$\widetilde {\{(x,y,z)\in \R^3 : x=y+z \}}=\{ (\alpha,\beta,\gamma)\in \widetilde \R^3 : \alpha=\beta+\gamma\},$$
\item[(v)] if $I \subset \R$ is an interval, then $\widetilde I \subset \widetilde \R$ is an interval,
\item[(vi)] if $S\subset \R$ is a singleton, then $\widetilde S \subset \widetilde \R$ is a singleton.
\end{enumerate}

Proof of $(i)$. The inclusion $\widetilde A \cup \widetilde B \subset \widetilde {A \cup B}$ is obvious. Conversely, take $\gamma \in \widetilde {A \cup B}$ with $\gamma$ defined and continuous on $(0;\epsilon)$. Assume $\gamma \notin \widetilde A$. Then $\gamma ^{-1}(A)$ is a finite union of points and intervals, therefore there exists 
$0<\epsilon_1<\epsilon$ such that $\gamma ((0;\epsilon_1)) \cap A=\emptyset$. Then $\gamma ((0;\epsilon_1)) \subset B$ and $\gamma \in \widetilde B$.

The equality $\widetilde {A \cap B}=\widetilde A \cap \widetilde B$ is obvious.

The inclusion $\widetilde {A \setminus B} \subset \widetilde A \setminus \widetilde B$ is obvious. Take $\gamma \in \widetilde A \setminus \widetilde B$, and assume $\gamma$ defined and continuous on $(0;\epsilon)$. Then $\gamma ^{-1}(B)\subset (0;\epsilon)$ is a finite union of points and intervals, that does not contain any interval of the form $(0;s)$ since $\gamma \notin \widetilde B$. Therefore $\gamma ((0;s)) \subset A \setminus B$. and $\gamma \in \widetilde {A \setminus B} $.

The equality $\widetilde {A\times B}=\widetilde A \times \widetilde B$ is obvious.

Proof of $(ii)$. Take $\gamma \in  \widetilde C$ and assume that $\gamma$ is defined and continuous on $(0;\epsilon)$, with $\gamma ((0;\epsilon)) \subset C$. Then $p\circ \gamma :(0;\epsilon) \to p(C)$ is a definable continuous function so that $\widetilde p (\gamma) \in \widetilde {p(C)}$. Therefore $\widetilde p (\widetilde C) \subset \widetilde {p(C)}$.\\
Conversely, take $\gamma \in \widetilde {p(C)}$, and choose $\epsilon>0$ so that $\gamma$ is a definable continuous function on $(0;\epsilon)$ with value in $p(C)$. Note that we can assume that $C$ is bounded, embedding if necessary $\R^{n+1}$ in $\mathbb S^{n+1}\subset \R^{n+2}$.
Then 
$$D=\{(t,c) \in (0;\epsilon)\times C: p(c)=\gamma(t)\}$$
is a definable set which boundary has a non-empty intersection with $\{0\}\times C$ since $C$ is bounded. By the curve selection lemma, there exists $\delta=(\delta_1,\delta_2) \in \widetilde D$, i.e. there exists $0<\epsilon_1<\epsilon$ such that $p(\delta_2(t))=\gamma (\delta_1(t))$ for $t \in (0;\epsilon_1)$. By Lemma \ref{repara} there exists $\alpha \in \wR$ defined on $(0;\epsilon_2)$, with $0<\epsilon_2<\epsilon_1$, such that $\delta_1 \circ \alpha (t)=t$ for $t\in (0;\epsilon_2)$. As a consequence $\delta_2 \circ \alpha \in \widetilde C$ satisfies $\widetilde p (\delta_2 \circ \alpha)=\gamma$, therefore $\gamma \in \widetilde p (\widetilde C)$.

The proofs of $(iii)$ and $(iv)$ are very similar, so let us prove $(iii)$. Take 
$$\gamma=(\gamma_1,\gamma_2,\gamma_3) \in \widetilde {\{ (x,y,z)\in \R^3 : x=yz\}}.$$
There exists $\epsilon>0$ such that $\gamma$ is defined on $(0;\epsilon)$ and for $t\in (0;\epsilon)$ the equality $\gamma_1(t)=\gamma_2(t) \gamma_3(t)$ holds. Therefore $\gamma_1=\gamma_2 \gamma_3$ and 
$$\widetilde {\{ (x,y,z)\in \R^3 : x=yz\}} \subset \{ (\alpha,\beta,\gamma)\in \widetilde \R^3 : \alpha=\beta \gamma\}.$$
If conversely $\alpha,\beta,\gamma\in \widetilde \R$ satisfies $\alpha=\beta \gamma$, then 
$$(\alpha,\beta,\gamma) \in \widetilde {\{ (x,y,z)\in \R^3 : x=yz\}}.$$

Proof of $(v)$. Assume $I=(a;b)$ with $a<b\in \R$. Then $\widetilde I=(\widetilde a;\widetilde b)$ where $\widetilde a, \widetilde b \in \wR$ denotes the constant germs equal to $a$ and $b$. Indeed, if $\gamma \in \widetilde I$, there exists $\epsilon>0$ such that $\gamma$ is defined on $(0;\epsilon)$ and $\gamma (t) \in (a;b)$ for $t\in (0;\epsilon)$. In particular $\widetilde a(t)=a<\gamma(t)<b=\widetilde b(t)$ for $t\in (0;\epsilon
)$ so $\gamma \in (\widetilde a;\widetilde b)$. The proof of the converse inclusion is similar.

Proof of $(vi)$. An element of $\widetilde S$ is a germ of continuous definable curve with image in $S=\{s\}$, so is equal to the function germ $\widetilde s$ constant equal to $s$.

\end{proof}

\begin{rmk} 

 The convex set given in example \ref{ex} is not a
  definable subset of $\widetilde {\mathbb R}=\Pui$
  since it can not be described by a finite union of intervals.

\end{rmk}


\section{Continuity}\label{co}
We prove in this section that a continuous definable map between definable sets induces a continuous map between the spaces of definable arcs. Using \L ojasiewicz inequality, we give a proof of this fact generalising the H\"older argument of Theorem \ref{main}. We propose another approach, studying more in details the transport of properties between the initial o-minimal structure and the new one constructed in section \ref{defin}. These results are a geometric interpretation of the model-theoretical aspects developed in \cite{Coste}, Chapter 5.

\subsection{\L ojasiewicz inequality in the closed and bounded case}\label{loja}

We recall \L ojasiewicz inequality in the context of o-minimal structures, not necessarily polynomially bounded. We refer to \cite{KK98} for the real case, and note that the real closed field case follows similarly.

Let $\R$ be a real closed field and and fix an o-minimal structure expanding it.

\begin{thm}\label{LI} Let $X\subset \R^n$ be a definable set, closed and bounded. Let $\phi_1,\phi_2:X \to \R$ be non negative continuous definable functions satisfying $\phi_2^{-1}(0)\subset \phi_1^{-1}(0)$. There exists a strictly increasing continuous definable function $\rho: [0,+\infty) \to [0,+\infty)$ such that
$$\forall x \in X, ~~\phi_2(x) \geq \rho \circ \phi_1(x).$$
\end{thm}

\begin{rmk} In the polynomially bounded case, we may choose $\rho$ of the form $\rho(s)=c|s|^r$, with $c,r>0$.
\end{rmk}

We apply \L ojasiewicz inequality to prove the continuity between the spaces of arcs in the closed and bounded case.

\begin{prop}\label{comp} Let $X$ and $Y$ be definable sets, with $X$ closed and bounded. Let $f:X\to Y$ be a continuous definable map. Then $\widetilde f: \widetilde X \to \widetilde Y$ is continuous. 
\end{prop}

\begin{rmk} In the polynomially bounded case, the map $f$ is H\"older by \L ojasiewicz inequality. In particular, in the case of the field of Puiseux series equipped with the $t$-adic topology, this result becomes clear.
\end{rmk}

\begin{proof}[Proof of Proposition \ref{comp}]
For $\gamma \in \widetilde X$, there exists $\epsilon$ such that $\gamma$ admits a representative $\gamma: [0,\epsilon) \to X$. Let $B\subset X$ be a closed neighbourhood of $\gamma (0)$ in $X$. Define two functions $\phi_1,\phi_2$ on $B\times B$ by
$$\phi_1(x_1,x_2)=|f(x_1)-f(x_2)|$$
and
$$\phi_2(x_1,x_2)=|x_1-x_2|.$$
Then $\phi_1$ and $\phi_2$ are continuous definable, $\phi_2^{-1}(0) \subset \phi_1^{-1}(0)$ and therefore by \L ojasiewicz inequality there exists a strictly increasing continuous definable function $\rho: [0,+\infty) \to [0,+\infty)$ such that
$$\forall (x_1,x_2) \in B\times B, ~~\phi_2(x_1,x_2) \geq \rho \circ \phi_1(x_1,x_2),$$
which means
$$\forall (x_1,x_2) \in B\times B, ~~|x_1-x_2| \geq \rho ( |f(x_1)-f(x_2)|).$$
As a consequence,
$$\forall (x_1,x_2) \in B\times B, ~~|f(x_1)-f(x_2)| \leq \rho ^{-1}(|x_1-x_2|),$$
and $\rho^{-1}$ is a definable continuous function with $\rho^{-1}(0)=0$. In particular, for any arc $\delta$ closed to $\gamma$, there exists a representative $\delta:[0,\epsilon) \to X$ of $\delta$ with value in $B$ and therefore
$$\forall t\in [0,\eps),~~ |f\circ \gamma(t)-f\circ \delta (t)| \leq \rho ^{-1}(|\gamma (t)-\delta (t)|),$$
which proves the continuity of $\widetilde f$.

\end{proof}

\subsection{A generalised \L ojasiewicz inequality}
We slightly generalise Theorem \ref{LI} in order to get rid of the boundedness assumption of $X$ is Proposition \ref{comp}.

\begin{prop}\label{L-loc} Let $X\subset \R^n$ be a locally closed definable set. Let $\phi_1,\phi_2:X \to \R$ be non negative continuous definable functions satisfying $\phi_2^{-1}(0)\subset \phi_1^{-1}(0)$. There exist a strictly increasing continuous definable function $\rho: [0,+\infty) \to [0,+\infty)$ and a positive continuous definable function $c:X\to \R$ such that $c_{|\phi_2^{-1}(0)}=1$ and 
$$\forall x \in X, ~~\phi_2(x) \geq c(x) \rho \circ \phi_1(x).$$
\end{prop}

\begin{proof} As $X$ is locally closed, we can find a definable homeomorphism mapping $X$ onto a closed definable set, therefore we can assume $X$ is closed. 
Let $U$ denote a definable neighbourhood of $\phi_1^{-1}(0)$ in $X$ such that $\phi_1|_{U}$ is bounded and $\phi_1|_{U-\phi_1^{-1}(0)}:U-\phi_1^{-1}(0)\to \phi_1(U-\phi_1^{-1}(0))$ is proper.  Such a neighbourhood $U$ always exists; for example, let $\phi_3$ denote the function on $X$ measuring the distance from a fixed point of $\phi_1^{-1}(0)$. 
Then the set $U=\{x\in X:\phi_1(x)\phi_3(x)\le1\}$ is convenient.

Let $Y$ denote the image of $U$ under the map $(\phi_1,\phi_2):X\to R^2$. 
Then $Y$ is a closed definable subset of a sufficiently small closed neighbourhood of the origin in $R^2$ by the properness condition. Moreover $Y\cap(\R\times\{0\})=\{(0,0)\}$ and therefore there exists a continuous definable function $\tau$ on $\Ima {\phi_1}_{|U}$ such that $\tau^{-1}(0)=\{0\}$ and $x_2\ge\tau(x_1)$ for $(x_1,x_2)\in Y$, so that $\phi_2\ge\tau\circ\phi_1$ on $U$. One can construct a strictly increasing continuous definable function $\rho:[0,+\infty) \to [0,+\infty)$ such that $\rho \leq \tau$ on $\phi_1|_{U}(U)$. Therefore $\phi_2 \geq \rho \circ \phi_1$ on $U$.

To conclude, it suffices to choose a positive continuous definable function $c$ so that $c_{|\phi_2^{-1}(0)}=1$ and $c \leq  \frac{\phi_2}{\rho \circ \phi_1}$ on $X\setminus \{\phi_2^{-1}(0)\}$.
\end{proof}

Let $M$ be a definable set in some o-minimal structure. We define $\widehat M$ by
$$\widehat M=\{\gamma\in\widetilde M:\lim_{t\to 0^+}\gamma(t)\in M\}.$$
If $h:M \to N$ is a continuous definable map between definable sets $M$ and $N$, define $\widehat h: \widehat M \to \widehat N$ as the restriction of $\widetilde h$ to $\widehat M$. Note that neither $\widehat M$ nor $\widehat h$ are definable.

\begin{cor}\label{non-comp} Let $X$ and $Y$ be definable sets, with $X\subset \R^n$ locally closed. Let $f:X\to Y$ be a continuous definable map. Then $\hat f: \hat X \to \hat Y$ is continuous. 
\end{cor}

\begin{proof} Using proposition \ref{L-loc} and proceeding similarly to the proof of Proposition \ref{comp}, we obtain that for $\gamma,\delta \in \hat X$, there exist $\epsilon>0$ such that 
$$\forall t\in [0,\eps),~~ |f\circ \gamma(t)-f\circ \delta (t)| \leq \rho ^{-1}(\frac{|\gamma (t)-\delta (t)|}{c(\gamma(t),\delta(t))}).$$
Let $U$ be defined by 
$$U=\{(x_1,x_2)\in X^2: ~~c(x_1,x_2)>\frac{1}{2}\}.$$
Fixing $\gamma$, there exists $l\in \mathbb N$ such that if $|\gamma(t)-\delta(t)|<t^l$ for $t\in [0,\eps)$, then $(\gamma(t),\delta(t)) \in U$ for $t\in [0,\eps)$. Therefore for $\delta \in B(\gamma, t^l)$ we obtain 
$$\forall t\in [0,\eps),~~ |f\circ \gamma(t)-f\circ \delta (t)| \leq \rho ^{-1}(2|\gamma (t)-\delta (t)|)$$
which proves the continuity of $\hat f$.
\end{proof}







\subsection{Transport of properties}\label{trans}

We develop now the other approach based on the extension of fields from a real closed field $\R$ to its field of germs of definable arcs $\widetilde \R$. First, we note that the closedness and boundedness of a definable set are preserved under the extension.

\begin{lemma}\label{closed} Let $A$ be a definable subset of a
  definable set $B$. Then $A$ is closed in $B$ if and only if $\widetilde A$ is closed in $\widetilde B$.
\end{lemma}

\begin{proof} It is equivalent to prove that $A$ is open in $B$ if and
  only if $\widetilde A$ is open in $\widetilde B$.

Assume $A$ is open in $B$. Take $\gamma \in \widetilde A$. There
exists $\epsilon >0$ such that $\gamma$ is well defined on
$(0;\epsilon)$ and $\gamma((0;\epsilon))\subset A$. The openness of $A$ enables to construct a strictly positive definable function germ by
$$r(t)=\sup \{s \in (0;1] : ~~|x-\gamma(t)|<s \Rightarrow x\in A\}$$
for $t \in (0;\epsilon)$. Reducing $\epsilon$ if necessary, we can suppose $r$ continuous. Then the open ball $B(\gamma, r)$ is included in $\widetilde A$ and therefore $\widetilde A$ is open.

Conversely, to prove that $A$ is open in $B$ if $\widetilde A$ is open
in $\widetilde B$, let us assume that $A$ is not open at $a\in
A$. Then $a$ belongs to the closure of $B\setminus A$, therefore by
the curve selection lemma there exists $\gamma :[0;\epsilon) \to B$
continuous definable with $\gamma ((0;\epsilon)) \subset B\setminus A$
and $\gamma (0)=a$. If $\widetilde A$ is open, let $B(\widetilde
a;r)=\{\delta \in \widetilde B: |\delta-\widetilde a|<r\} \subset \widetilde A$
be an open neighbourhood of the constant germ $\widetilde a$ equal to
$a$, with $r:(0;\epsilon_1) \to \R$ a strictly positive definable
germ, and $0<\epsilon_1<\epsilon$. We get a contradiction if $\gamma
\in B(\widetilde a;r)$. Otherwise, we reparametrize $\gamma$ as
follows. Define $\alpha : (0;\epsilon_1)  \to \R$ by
$$\alpha (t)=\sup \{s\in (0;t) : |\gamma (s)-a|<\frac{r(t)}{2}\}$$
Then $\alpha$ is definable and strictly positive since $\gamma (0)=a$ and $\gamma$ is continuous. We can moreover assume that $\alpha$ is continuous, decreasing $\epsilon_1$ if 
necessary. Then $\gamma \circ \alpha ((0;\epsilon_1)) \subset B\setminus A$ and $\gamma \circ \alpha \in B(\widetilde a;r)$.
\end{proof}

\begin{lemma}\label{bounded} Let $A$ be a definable set in
  $\R^n$. Then $A$ is bounded if and only if $\widetilde A$ is bounded in $\widetilde \R^n$.
\end{lemma}

\begin{proof} Assume $A$ is bounded in $\R^n$. Then there exist $c_1,c_2 \in \R$ such that $A\subset [c_1;c_2]^n$. For $\gamma \in \widetilde A$, there exists $\epsilon$ such that $\gamma (t) \in  [c_1;c_2]^n$ for $t \in (0;\epsilon)$. In particular $\gamma \in [\widetilde c_1;\widetilde c_2]^n$ and $\widetilde A$ is bounded.


Assume that $A$ is not bounded whereas $\widetilde A$ is bounded, namely there exists $b \in \widetilde R$ such that $|\gamma| < b$ for any $\gamma \in \widetilde A$. Since $A$ is not bounded, there exists by Theorem \ref{csl} a curve $\gamma:(0,\eps)\to A$ which is not bounded. By reparametrisation via Lemma \ref{repara}, we may assume that $|\gamma(t)|=1/t$. As a consequence $\beta=\gamma \circ \frac{1}{b}$ is defined on $(0,\epsilon')$ for some $0<\epsilon'<\epsilon$, satisfies $\beta \in \widetilde A$ and $|\beta|=|b|$ and therefore $\beta$ contradicts our assumptions.
\end{proof}

The injectivity and surjectivity of a definable map are also preserved under the extension of fields.

\begin{lemma} Let $f: X \to Y$ be a definable map. Then $f$ is injective if and only if $\widetilde f$ is injective. Similarly, $f$ is surjective if and only if $\widetilde f$ is surjective.
\end{lemma}

\begin{proof} Assume $f$ is injective. Let $\gamma_1,\gamma_2 \in
  \widetilde X$ satisfies $\widetilde f (\gamma_1)=\widetilde f (\gamma_2)$. There exists $\epsilon >0$ such that $\gamma_1$ and
  $\gamma_2$ are defined on $(0;\epsilon)$ and then for $t\in
  (0;\epsilon)$ we have $f(\gamma_1(t))=f(\gamma_2(t))$. Therefore
  $\gamma_1(t)=\gamma_2(t)$ for $t\in
  (0;\epsilon)$ by injectivity of $f$ and $\widetilde f$ is injective.

Assume $\widetilde f$ injective. Let $x_1,x_2 \in X$ satisfy
$f(x_1)=f(x_2)$. If $\widetilde x_i$ denotes the constant curve equal
to $x_i$, for $i\in \{1,2\}$, then $\widetilde f (\widetilde
x_1)=\widetilde f (\widetilde x_1)$. Therefore $\widetilde
x_1=\widetilde x_2$ and thus $x_1=x_2$ and $f$ is injective.

Assume now that $f$ is surjective. Let $\gamma \in \widetilde Y$ and
take $\epsilon >0$ such that $\gamma$ is defined and continuous on
$(0;\epsilon)$. Consider the set
$$A=\{(x,t)\in f^{-1}(\gamma((0;\epsilon)))\times (0;\epsilon) : f(x)=\gamma(t)\}.$$
Then $A$ is a definable set, and we may assume $A$ is bounded
(embedding the ambient space $\R^n$ of $X$ in $\mathbb S^n$ if
necessary), so that $(\overline A\setminus A) \cap (\R^n \times
\{0\})$ is not empty. By the curve selection lemma, there exists
$\delta =(\delta_1,\delta_2) \in \widetilde A$, defined on
$(0;\epsilon')$ where $\epsilon'<\epsilon$, such that $\delta
((0;\epsilon'))\subset  A$ and $\delta$ admits a limit at $0$, denoted
by $\delta(0)$, and $\delta(0) \in (\overline A \setminus A)\cap (\R^n \times \{0\})$. Then $f\circ \delta_1 (t)=\gamma (\delta_2(t))$ by definition of $A$. Denote by $\alpha \in \wR$ a continuous definable function germ such that $\delta_2\circ \alpha (t)=t$, obtained by Lemma \ref{repara}. Then $\delta_1 \circ \alpha \in \widetilde X$ is an inverse of $\gamma$ by $\widetilde f$ and $\widetilde f$ is surjective.

Assume finally that $\widetilde f$ is surjective. Take $y\in
Y$. Then the constant curve $\widetilde y$ equal to $y$ has a pre-image
$\gamma$ by $\widetilde f$. There exists $\epsilon >0$ such that
$\gamma$ is well defined on $(0;\epsilon)$,
$\gamma((0;\epsilon))\subset X$ and
$f(\gamma(t))=y$. Therefore $f$ is surjective.
 
\end{proof}


\subsection{Continuity}\label{conti}
Let $\R$ be a real closed field and consider an o-minimal structure
expanding $\R$. A definable mapping $f: X \to Y$ between definable sets
$X,Y$ is said continuous if it is continuous for the Euclidean
topology inherited to the real closed field.

\begin{prop}\label{cont} Let $f: X \to Y$ be definable and $Y$ be bounded. Then
  $f$ is continuous if and only if $\graph f$ is closed in $X \times
  \overline Y$.
\end{prop}

\begin{proof} Assume $f$ is continuous. Let $(a,b)$ be in the closure of the graph of $f$ in $X\times \overline Y$. By the curve selection lemma there exists a continuous curve $\gamma=(\gamma_1, \gamma_2): [0;1] \to X\times \overline Y$ with $\gamma (0)=(a,b)$ and $\gamma ((0;1]) \subset \graph f$. Then for $t\in (0;1]$ we have $f\circ \gamma_1(t)=\gamma_2(t)$. By continuity we obtain, passing to the limit as $t$ goes to zero, that $f(a)=b$. In particular the graph of $f$ is closed in $X\times \overline Y$.

Conversely, assume $f$ is not continuous at $a\in X$. For $\epsilon >0$, define $\delta_{\epsilon}$ to be the supremum on $x\in X$ with $|x-a|\leq \epsilon$ of $|f(x)-f(a)|$. Then the pair $(0,a)$ is in the closure of the definable set
$$\cup_{\epsilon >0} \{\epsilon \}\times \{x\in X: |x-a| \leq \epsilon \textrm{ and } |f(x)-f(a)| \geq \frac{\delta_{\epsilon}}{2}\}.$$
By the o-minimal curve selection lemma, there exists a continuous
definable curve $\gamma: [0;1] \to X$ such that $\gamma(0)=a$ and $f \circ \gamma $ is not continuous at $0$.

Denote by $\Gamma$ the image $\gamma ((0;1])$ of $\gamma$ and by $A$ the set
$$\{(x,f(x)) : x \in \Gamma \} \subset X\times Y.$$
Then $\overline A \setminus A$ is not empty because $Y$ is bounded. It is then equal to a point of the form $(a,b)$ with $b \neq f(a)$, therefore the graph of $f$ is not closed in $X\times \overline Y$.
\end{proof}

\begin{cor}\label{coro-main} Let $f: X \to Y$ be a continuous definable map. Then
  $\widetilde f: \widetilde X \to \wY$ is continuous.
\end{cor}

\begin{proof}
Without loss of generality we can assume that $Y$ is bounded (as
$Y\subset \R^n \subset \mathbb S^{n}\subset \R^{n+1}$). Now the set
$\widetilde Y$ is bounded by Lemma \ref{bounded}, and moreover
$\overline {\widetilde Y}=\widetilde {\overline Y}$. Finally the graph of $\widetilde f$ is closed in $\widetilde X \times \overline {\widetilde Y}$ by Lemma \ref{closed} since the graph of $f$ is closed in $X \times
  \overline Y$ by Proposition \ref{cont}.
\end{proof}

\begin{rmk} In a model-theoretic approach, we can deduce the continuity of $\widetilde f$ using the fact that the extension $\widetilde R$ is an elementary extension of $\R$ (cf. \cite{Coste}, p58).
\end{rmk}

\begin{ex}\label{exam} As a illustrative example, consider the Whitney umbrella $W$ in $\mathbb R^3$ defined by $x^2z=y^2$. Its two-dimensional part is the image of $\mathbb R^2$ by the map $h:(u,v)\mapsto (u,uv,v^2)$, giving a resolution of the singularities of $W$. Any positive element of the $z$-axis has two pre-images by $h$, and so $h$ defines a semi-algebraic bijection between the semi-algebraic sets $(\mathbb R^2\setminus \{u=0\}) \cup \{(0,0)\}$ and $(W\setminus \{x=y=0\}) \cup \{(0,0,0)\}$, which are not even locally closed. Corollary \ref{coro-main} enables to assert than nevertheless, $h$ induces a homeomorphism between the associated spaces of arcs.
\end{ex}

We recover the results in section \ref{res-t} as a consequence of Corollary \ref{coro-main} by considering the field of real numbers with the o-minimal structure given by globally subanalytic sets. For instance for Theorem \ref{thm-R}:

\begin{cor} Let $h: \mathbb R^n \to \mathbb R^n$ be a
  subanalytic homeomorphism. Then the map $\widetilde h:\mathbb R \{t^{1/\infty}\}^n \to \mathbb R \{t^{1/\infty}\}^n$ is a homeomorphism with respect to the $t$-adic topology. If moreover we assume that $h$ and $h^{-1}$ are arc-analytic, then $h_*:\mathcal A_0(\mathbb R^n) \to \mathcal A_0(\mathbb R^n)$ is a homeomorphism.
\end{cor}

\begin{proof} The continuity of $\widetilde h$ follows from Corollary \ref{coro-main}. Moreover $h_*$ is well-defined by arc-analyticity of $h$ and $h^{-1}$, and so $h_*$ is continuous by restriction.
\end{proof}

\begin{prop}\label{unif} Let $X$ and $Y$ be bounded and closed definable sets and $f:X \to Y$ be a continuous definable function. Then $f$ is uniformly continuous.
\end{prop}

\begin{proof} Let $\epsilon \in \R$ such that $\epsilon >0$. As $f$ is continuous, we may define for $x \in X$ a definable function $g$ on $X$ by
$$g(x)=\sup \{ s \in (0;1]: a\in X, |x-a|<s \Rightarrow |f(x)-f(a)|<\epsilon \}.$$
We want to prove that there exists $r>0$ such that $g\geq r$.
Assume it is not true. Then
$$A=\{(t,x)\in (0;1]\times X : g(x)< t\}$$ 
is a non empty definable subset of $[0;1]\times X$, with a non empty
boundary. Therefore there exists a continuous definable function germ
$\delta=(\delta_1,\delta_2) :[0;\epsilon) \to [0;1]\times X$ with $\delta((0;\epsilon))\subset A
$ and $\delta_1(0)=0$ by the curve selection lemma. Note that $\delta_2 (t)$ admits a limit $x_0$ as $t$ goes to zero since $X$ is bounded, and $x_0$ belongs to $X$ since $X$ is closed.

For $x\in X$ such that $|x-x_0|<g(x_0)/2$ and $y\in X$ such that $|y-x|<g(x_0)/2$, we have $|y-x_0|<g(x_0)$ so that $|f(y)-f(x_0)|<\epsilon$. In particular $g(x)\geq g(x_0)/2$. For $t$ small enough we have $|\delta_2(t)-x_0|<g(x_0)/2$ since $\delta_2$ goes to $x_0$, so that 
$$g(x_0)/2 \leq g\circ \delta_2(t) < \delta_1(t)$$ 
for $t$ small enough, in contradiction with the fact that $\delta_1(0)=0$.
\end{proof}

We recover by this way Theorem \ref{thm-comp}.

\begin{cor}\label{cor-last} Let $M\subset \mathbb R^m$ and $N\subset \mathbb R^n$ be compact analytic manifolds and $h:M \to N$ be a subanalytic homeomorphism. Then $\widetilde h: \widetilde M \to \widetilde N$ is a uniformly continuous homeomorphism. If moreover $h$ and $h^{-1}$ are arc-analytic, then $h_*: \mathcal A(M) \to \mathcal A(N)$ is a uniformly continuous homeomorphism.
\end{cor}

\begin{proof} The uniform continuity of $\widetilde h$ follows from Corollary \ref{coro-main} and Proposition \ref{unif}. Then $h_*$ is well-defined by arc-analyticity of $h$ and $h^{-1}$, and finally $h_*$ is uniformly continuous by restriction.
\end{proof}

Actually, we can even drop the compactness condition in Corollary \ref{cor-last} as follows. 

\begin{prop}\label{prop-fin} Let $M\subset \mathbb R^m$ and $N\subset \mathbb R^n$ be globally subanalytic manifolds and $h:M \to N$ be a  globally subanalytic homeomorphism. Then $\widehat h: \widehat M \to \widehat N$ is a uniformly continuous homeomorphism. If moreover $h$ and $h^{-1}$ are arc-analytic, then $h_*: \mathcal A(M) \to \mathcal A(N)$ is a uniformly continuous homeomorphism.
\end{prop}

The proof is a direct consequence of the following lemma.
\begin{lemma} Let $f:X\to Y$ be a continuous definable map between closed definable subsets of $\R^n$. Then $\hat f:\hat X \to \hat Y$ is uniformly continuous.
\end{lemma}

\begin{proof} Let $a\in \widetilde R$ be an element larger than any element of $\R$. since $\widetilde X \cap [-a,a]^n$ is closed and bounded, the result follows from Proposition \ref{unif}.
\end{proof}

\enddocument